\documentclass{amsart}
\usepackage{comment}

\newtheorem{thm}{Theorem}[section]
\newtheorem{cor}[thm]{Corollary}
\newtheorem{lem}[thm]{Lemma}
\newtheorem{prop}[thm]{Proposition}
\theoremstyle{definition}
\newtheorem{defn}[thm]{Definition}
\theoremstyle{remark}

\numberwithin{equation}{section}

\newcommand{\M}{\mathcal{M}}

\newcommand{\Real}{\mathbb{R}}

\renewcommand{\div}{\text{div}}
\newcommand{\Hyperbolic}{\mathbb{H}}

\begin{document}

\title[Translating graphs by mean curvature flow]
 {Translating graphs by mean curvature flow}

\author{Leili Shahriyari}
\address{Department of Mathematics, Johns Hopkins University, 3400 N Charles St, Baltimore, MD 21218}
\email{shahriya@math.jhu.edu}







\dedicatory{}



\begin{abstract}
The aim of this work is studying translating graphs by mean curvature flow in $\Real^3$. We prove non-existence of complete translating graphs over bounded domains in $\Real^2$. Furthermore, we show that there are only three types of complete translating graphs in $\Real^3$; entire graphs, graphs between two vertical planes, and graphs in one side of a plane. In the last two types, graphs are asymptotic to planes next to their boundaries. We also prove stability of translating graphs and then we obtain a pointwise curvature bound for translating graphs in $\Real^3$.
\end{abstract}

\maketitle

\section{Introduction}

 Mean curvature flow evolves hypersurfaces in the unit normal direction with speed equal to the mean curvature at each point. It is the steepest descent flow for the area functional. In particular, minimal hypersurfaces are stationary solutions. In other words, a family of smoothly embedded hypersurfaces $(\M_t)_{t\in I}$ moves by mean curvature if
\begin{equation}\label{mcf}
\frac{\partial x}{\partial t}=\vec{H}(x),
\end{equation}
for $x\in\M_t$ and $t\in I$, $I\subset\Real$ an open interval. Here $\vec{H}(x)$ is the mean curvature vector at $x\in\M_t$.

The evolution equation (\ref{mcf}) can develop singularities in finite time $T$, which are classified into two types according to the rate at which the maximal curvature, $\displaystyle\max_{\M_t}|A(t)|$, tends to infinity for $t\to T$. Here $|A(t)|$ is the second fundamental form of $\M_t$. By proving the monotonicity formula, Huisken \cite{Hui90} showed that the flow is asymptotically self-similar near a given type-I singularity and, thus, is modeled by self-shrinking solutions of the flow. However, the examples of convergence in \cite{AV95, AV97} indicate that type-II singularities are modeled by translating surfaces. Also, Huisken and Sinestrari \cite{GC} proved that if the initial surface $\M_0$ has nonnegative mean curvature, then any limiting flow of a type-II singularity has convex surfaces $\widetilde{\M}_t$, $t\in\Real$. Furthermore, either $\widetilde{\M}_t$ is a strictly convex translating soliton or (up to rigid motion) $\widetilde{\M}_t=\Real^{n-k}\times\Sigma^k_t$, where $\Sigma^k_t$ is a lower dimensional strictly convex translating soliton in $\Real^{k+1}$. The proof of this theorem used an important theorem of Hamilton \cite{Hamilton}, which states that any strictly convex eternal solution to the mean curvature flow, where the mean curvature assumes its maximum value at a point in space-time must be a translating solution.

Note that solitons generally move by symmetries of Euclidean space, either scaling symmetries or translations, that means for these surfaces, we have
\begin{equation}
\frac{\partial x}{\partial t}=\vec{C}+\vec{v},
\end{equation}
where $\vec{C}$ is the velocity vector of the translation and $\vec{v}$ is a vector field tangent to the surface $\M_t$. When $\vec{C}=-\frac{x}{2}$ or $\vec{C}=e_{n+1}$, we will respectively get self similar shrinkers and vertically translating surfaces in $\Real^{n+1}$. If $\vec{C}=Ce_{n+1}$, and taking the inner product with the unit normal vector $\upsilon$, we obtain the following equation for translating surfaces
\begin{equation}
H=C\langle e_{n+1},\upsilon \rangle.
\end{equation}

Translating graphs by mean curvature flow are translating surfaces that can be viewed as a graph of a function over a domain. Let the graph of the function $u=u(x)$ be a translating graph by the mean curvature flow. Since $H=\div(\frac{\nabla u}{\sqrt{1+|\nabla u|^2}})$, the graph of $u$ is a vertically translating graph with constant speed $C$ if and only if $u$ is a solution to the  following equation
 \begin{equation}\label{te}
 \div\left(\frac{\nabla u}{\sqrt{1+|\nabla u|^2}}\right)=\frac{C}{\sqrt{1+|\nabla u|^2}}.
  \end{equation}

In this work, we study translating graphs by mean curvature flow in $\Real^3$. First of all, we prove these graphs are stable minimal surfaces in a certain conformal metric. Therefore, the curvature estimate for minimal surfaces gives us a curvature estimate for translating graphs. Using this curvature estimate, we obtain non-existence of complete translating graphs over bounded domains in $\Real^2$. Furthermore, we show that there are only three types of complete translating graphs in $\Real^3$; entire graphs, graphs between two vertical planes, and graphs in one side of a plane. In the last two types, graphs are asymptotic to planes next to their boundaries.

The author would like to thank Bill Minicozzi and Xuan Hien Nguyen for many helpful discussions of this paper.


\section{Stability of translating graphs}

Colding and Minicozzi proved self similar shrinkers are minimal surfaces in a certain conformal metric \cite{CM, CM1}. In this section, using the same method, we prove the similar statement for translating graphs, i.e. translating graphs are stable minimal hypersurfaces in a conformal metric.

From now on without loss of generality we assume that the speed $C=-1$. Let $p$ be a fix point in $\Real^3$, and $r>0$, then we define the functional $F$ on a hypersurface $\Sigma\subset\Real^{3}$ by
	\begin{equation}
	\label{eq:functional}
	F(\Sigma)=\displaystyle\int_{\Sigma\cap B_r(p)} e^{x_{3}} d\mu.
	\end{equation}

In the following arguments, we prove translating surfaces are stable minimal hypersurfaces in $\Real^{3}$ with respect to the conformal metric $g_{ij}=e^{x_{3}}\delta_{ij}$.

\begin{lem}
If $x'=\eta \upsilon$ is a compactly supported normal variation of a hypersurface $\Sigma\subset\Real^{3}$ and $s$ is the variation parameter, then
	\[
	\frac{\partial}{\partial s} F(\Sigma_s)=\displaystyle\int_{\Sigma\cap B_r(p)} \eta\left(H+\langle e_{3},\upsilon\rangle\right)e^{x_3} d\mu.
	\]
\end{lem}

\begin{proof}
By the first variation formula we have $(d\mu)'=\eta H d\mu$. Also the $s$ derivative of $\log(e^{x_{3}})$ is given by $\eta\langle e_{3}, \upsilon\rangle$. Thus we have the lemma.
\end{proof}

 In this part $\Sigma \subset \Real^3$ is a smooth embedded hypersurface; $\Delta$ and $\nabla$ are the (submanifold) Laplacian and gradient respectively, on $\Sigma$. Let the graph $\widetilde\Sigma$ be the graph $\Sigma$ respect to the conformal metric $g$. We know that the hypersurface $\widetilde{\Sigma}$ subset of Riemannian manifold $(\Real^3,g)$ is stable minimal hypersurface if and only if the second derivative of the area functional for all normal variations of $\Sigma$ is positive at $s=0$.

Now we are defining the second order operator $L$ by
	\begin{equation*}
	L\eta =\Delta \eta +|A|^2 \eta+\langle e_{3}, \nabla \eta\rangle.
	\end{equation*}

\begin{defn} \label{Def:Lstable}
We say a translating surface is $L$-stable, if for any compactly supported function $\eta$ we have
	\begin{equation}
	\label{eq:L-stable}
	\int_\Sigma \eta L\eta e^{x_{3}}\leq 0.
	\end{equation}
\end{defn}

The linear operator $L$ is associated to normal perturbations of $H + \langle e_{3}, \upsilon \rangle$. The function $H + \langle\vec e_{3}, \upsilon \rangle$ is invariant under translations in $\Real^2$, therefore $\langle \mathbf{v}, \upsilon\rangle$ is in the kernel of $L$ for any constant vector $\mathbf{v}$.

\begin{prop}\label{ms}
If the translating graph $\Sigma$ in $\Real^3$ with respect to the Euclidean metric is $L$-stable, then it is a stable minimal hypersurface in $\Real^3$ with respect to the conformal metric $g_{ij}=e^{x_3}\delta_{ij}$.
\end{prop}
\begin{proof}
Let $\eta$ be a smooth compactly supported function over $\Sigma$. 
\begin{eqnarray}
&&\frac{\partial^2}{\partial s^2} F(\Sigma_s) \big|_{s=0}=\frac{\partial}{\partial s} \left[\displaystyle\int_{\Sigma\cap B_r(p)} \eta\left(H+\langle e_{3},\upsilon\rangle\right)e^{x_3} d\mu\right]_{s=0} =  \notag \\
&& \displaystyle\int_{\Sigma\cap B_r(p)} \left[ \left( \frac{\partial}{\partial s}\left(H+\langle e_{3},\upsilon\rangle\right)\right) \eta e^{x_3} d\mu
+ \left(H+\langle e_{3},\upsilon\rangle\right) \frac{\partial}{\partial s} \left(\eta e^{x_3} d\mu\right)  \right]_{s=0}.
\end{eqnarray}

Since $\Sigma$ is translating graph, at $s=0$,  $H+\langle e_{3},\upsilon\rangle=0$. So we have

\begin{eqnarray}
\frac{\partial^2}{\partial s^2} F(\Sigma_s) \big|_{s=0}= \displaystyle\int_{\Sigma\cap B_r(p)} \left(\frac{\partial}{\partial s}\left(H+\langle e_{3},\upsilon\rangle\right)\right) \eta e^{x_3} d\mu \big|_{s=0}
\end{eqnarray}

From Lemma A.2 in \cite{CM}, we have
\begin{eqnarray}
\frac{\partial \upsilon}{\partial s}\big|_{s=0}&=& - \nabla \eta, \\
\frac{\partial H}{\partial s}\big|_{s=0}&=&- \Delta \eta - |A|^2 \eta.
\end{eqnarray}

By definition of the operator $L$, we obtain

\begin{eqnarray}
\frac{\partial^2}{\partial s^2} F(\Sigma_s)\big|_{s=0}= - \displaystyle\int_{\Sigma\cap B_r(p)} \eta L\eta e^{x_3}.
\end{eqnarray}

\end{proof}

\begin{lem}
\label{lem:L-kernel}
For every constant vector $\mathbf{v}$, we have $L\langle \mathbf{v}, \upsilon\rangle=0$.
\end{lem}

\begin{proof} We give a computational proof.
Let $\gamma_i$ be an orthonormal frame for $\Sigma$ and set $\xi=\langle \mathbf{v},\upsilon\rangle$. Working at a fixed point $P$ and choosing the frame $\gamma_i$, so that $\nabla_{\gamma_i}^T \gamma_j(P)=0$, differentiating gives at $P$ that
	\begin{equation}
	\label{e1}
	\nabla_{\gamma_i} \xi=\langle \mathbf{v}, \nabla_{\gamma_i}\upsilon\rangle=-a_{ij}\langle \mathbf{v}, \gamma_j\rangle.
	\end{equation}
Using Codazzi equation at $P$, we have
	\begin{equation*}
	\nabla_{\gamma_k}\nabla_{\gamma_i} \xi=-a_{ik,j}\langle \mathbf{v}, \gamma_j\rangle-a_{ij}\langle \mathbf{v}, a_{jk}\upsilon\rangle.
	\end{equation*}
Taking the trace gives
	\begin{equation}
	\label{e2}
	\Delta \xi=\langle \mathbf{v}, \nabla H \rangle-|A|^2 \xi.
	\end{equation}
Notice that
	\begin{align*}
	\nabla_{\gamma_i} H&=-\nabla_{\gamma_i}\langle e_{3},\upsilon\rangle \notag\\
	&=-\langle\nabla_{\gamma_i} e_{3}, \upsilon\rangle - \langle e_{3},\nabla_{\gamma_i} \upsilon\rangle \notag\\
	&= a_{ij}\langle e_{3},\gamma_j\rangle.
	\end{align*}
Therefore, using (\ref{e1}) we have
	\begin{equation}
	\label{e3}
	\langle\nabla H, \mathbf{v}\rangle=a_{ij}\langle e_{3},\gamma_j\rangle\langle \gamma_i,\mathbf{v}\rangle=-\langle e_{3},\nabla \xi\rangle.
	\end{equation}
Thus, by \eqref{e2} and \eqref{e3} we get
	\begin{equation*}
	L \xi=\Delta \xi+\langle\nabla \xi,e_{3}\rangle+|A|^2 \xi=0. \qedhere
	\end{equation*}

\end{proof}

\begin{thm}\label{Lstable}
Translating graphs in $\Real^3$ are $L$-stable.
\end{thm}

\begin{proof}
Since $\Sigma$ is a graph, there is a unit vector $\mathbf{v}$ in $\Real^{3}$ so that $\langle \mathbf{v}, \upsilon(x)\rangle\neq 0$ for all $x\in\Sigma$. We define the function $\xi$ on $\Sigma$ by
	\begin{equation*}
	\xi(x) = \langle \mathbf{v}, \upsilon(x)\rangle.
	\end{equation*}
It follows that  $0<\xi \leq 1$ and, by Lemma \ref{lem:L-kernel}, that $L \xi = 0$. Given any smooth compactly supported function $\eta$ on $\Sigma$, the function $\phi = \eta \xi$ satisfies
	\begin{align}
	L(\phi)&=\eta L \xi + \xi \left( \Delta \eta + \langle e_{3} , \nabla \eta \rangle \right) + 2 \langle \nabla \eta , \nabla \xi \rangle, \notag\\
	\label{eq:L-phi}	 &= \xi \left( \Delta \eta + \langle e_{3} , \nabla \eta \rangle \right) + 2 \langle \nabla \eta , \nabla \xi \rangle .
	\end{align}
Using Stokes' theorem with $\frac{1}{2} \div \left( \xi^2 \nabla \eta^2 e^{x_{3}} \right)$, we obtain
	\begin{equation}
	\label{eq2}
	\int \frac{1}{2} \langle \nabla \eta^2 , \nabla \xi^2 \rangle e^{x_{3}} = -\int \xi^2 \left( \eta \Delta \eta + |\nabla\eta|^2+\eta \langle e_{3} , \nabla\eta \rangle\right) e^{x_{3}}
	\end{equation}
Applying \eqref{eq:L-phi} and \eqref{eq2}, we obtain
	\begin{equation}\label{eq3}
	\int \phi L(\phi) e^{x_{3}} = -\int \xi^2|\nabla\eta|^2 e^{x_{3}} \leq 0 \qedhere
	\end{equation}
When $\Sigma$ is graphical, we have a direction $\omega$ for which $\xi>0$ for all $x\in \Sigma$. Given a smooth compactly supported function $\phi$, we take $\eta(x) := \phi(x)/ \xi(x)$. This means that \eqref{eq3} is true for any compactly supported function $\phi$, which is the definition of $L$-stability.
\end{proof}

\section{Curvature estimate}

Colding and Minicozzi proved compactness theorem for self similar shrinkers in \cite{CM1}. In this section, we obtain the similar theorem for translating graphs using a point wise curvature estimate for translating graphs by mean curvature flow in $\Real^3$. For reaching this goal, we state theorem 2.10 in \cite{CM2}, which is Schoen curvature estimate for two dimensional minimal hypersurfaces $\Sigma$ immersed in Reiemannian manifold $\mathbb{M}^3$ with sectional curvature $K_{\mathbb{M}}$\cite{Schoen-1}. For $x\in \mathbb{M}$, $B_s(x)$ denotes the extrinsic geodesic ball with radius $s$ and center $x$. Similarly, For $x\in\Sigma$, $B^\Sigma_s(x)\subset\Sigma$ denotes the intrinsic geodesic ball and $r$ the intrinsic distance to $x$.

\begin{thm}[Schoen Curvature estimate \cite{Schoen-1}, Colding-Minicozzi \cite{CM2}]\label{CurvatureBound}
If $\Sigma^2\subset \mathbb{M}^3$ is an immersed stable minimal surface with trivial normal bundle and $B_{r_0}=B_{r_0}^\Sigma(x)\subset\Sigma\setminus \partial \Sigma$, where $|K_\mathbb{M}|\leq k^2$ and $r_0 < \rho_1(\pi/k,k)$ (with $\rho_1<min\{\pi/k,k\}$), then for some $C=C(k)$ and all $0<\sigma\leq r_0$,
\begin{eqnarray}
\displaystyle\sup_{B^\Sigma _{r_0 - \sigma}} |A|^2 \leq C\sigma^{-2}.
\end{eqnarray}
\end{thm}

For applying this theorem to obtain the curvature estimate, we need to compute the sectional curvature of $\Real^3$ respect to the conformal metric $g$. By doing some simple computations, we get for every $1\leq i,j\leq 2$, $K_{ij}=-\frac{1}{4}e^{-x_3}$, and $K_{i3}=K_{3i}=0$.

\begin{thm}\label{CurvatureBound2}
Let $\Sigma^2\subset \Real^3$ be a complete translating graph in mean curvature flow, if $B_{r_0e}^\Sigma(p)\subset(\Sigma\cap B_1(p))\setminus \partial (\Sigma\cap B_1(p))$, and $r_0e^{1/2} < \rho_1(\pi e^{-1},e)$, then for some $C$ and all $0<\sigma\leq r_0$,
\begin{eqnarray}
\displaystyle\sup_{B^\Sigma _{r_0 - \sigma}} |A|^2 \leq C \sigma^{-2}.
\end{eqnarray}
\end{thm}

\begin{proof}
For point $p\in\Sigma$, let $B_1$ be the Euclidean unit ball of radius $1$ and center $p$ in $\Real^3$. Define $\hat\Sigma=\Sigma\cap B_1$, note that $\hat\Sigma$ is immersed submanifold in $B_1$ and $|\hat{A}|(p)=|A|(p)$. Now let $\tilde{B}_1$ be $B_1$ with respect to the metric $g_{ij}=e^{x_3-p_3}\delta_{ij}$. From theorem \ref{Lstable}, $\hat\Sigma$ is stable minimal hypersurface in $\tilde{B}_1$. Note that we only multiplied the metric $g_{ij}$ in the theorem by a constant $e^{-p_3}$, which wont change the results of theorem.

Let the distance $d$ be the distance corresponding to Euclidean metric and $\tilde{d}$ the distance corresponding to the conformal metric. Note that $B^{\tilde{\Sigma}}_{r_0}=\{x \in \Sigma: \ \tilde{d}(x,p)<r_0\}$, where $\tilde{d}(x,p)$ is the infimum of the length of geodesic curves connecting two points $p$ and $x$ in $\tilde{\Sigma}$. For $x=(x_1,x_2,x_3)\in\Sigma \cap B_1$, define minimizing geodesic $\gamma:[0, 1] \to \Sigma\cap B_1$ in Euclidean metric, connecting $p$ and $x$ in $\Sigma$, such that $\gamma(0)=p$ and $\gamma(1)=x$. We have
\begin{eqnarray}
\tilde{d}(x,p)&\leq& \displaystyle \int_0^1 ||\gamma'(t)|| dt \notag\\
 &=&\displaystyle \int_0^1 \sqrt{\langle \gamma'(t),\gamma'(t)\rangle} dt \notag \\
&=&\displaystyle \int_0^1 e^{\frac{\gamma_3(t)-p_3}{2}}|\gamma'(t)| dt \notag \\
&\leq& e^{1/2} d(x,p).
\end{eqnarray}

This implies if $x\in B^\Sigma_r(p)\subset B_1$, then $x\in B^{\tilde{\Sigma}}_{re^{1/2}}(p)$. Now define minimizing geodesic $\tilde\gamma:[0, 1] \to \tilde\Sigma\cap \tilde{B}_1$ connecting $p$ and $x$ in $\tilde\Sigma \cap \tilde{B}_1$, such that $\gamma(0)=p$ and $\gamma(1)=x$. Using Cauchy Schwartz inequality we have
\begin{eqnarray}
e^{-1/2}d(x,p)&\leq& \displaystyle \int_0^1 e^{-1/2}|\tilde\gamma'(t)| dt \notag \\
&\leq& \displaystyle \int_0^1 e^{(\tilde\gamma_3(t)-p_3)/2}|\tilde\gamma'(t)| dt \notag \\
&=&\displaystyle \int_0^1 \sqrt{\langle \tilde\gamma'(t),\tilde\gamma'(t)\rangle} dt \notag \\
&=&\displaystyle \int_0^1 ||\gamma'(t)|| dt=\tilde{d}(x,p).
\end{eqnarray}

This implies if $x\in B_r^{\tilde\Sigma}(p)\subset \tilde{B}_1$, then $x\in B^\Sigma_{re^{1/2}}(p)$.
Note that if $x \in B_{r_0e^{1/2}}^{\tilde{\Sigma}}(p)$, then $x\in B_{r_0e}^\Sigma(p)\subset(\Sigma\cap B_1(p))\setminus \partial (\Sigma\cap B_1(p))$ which implies that $x\in\hat\Sigma\setminus \partial \hat\Sigma$. Also if $x\in B^\Sigma _{r_0 - \sigma}$, then $x\in B_{e^{1/2}(r_0 - \sigma)}^{\tilde\Sigma}$.

Since sectional curvature of $\tilde{B}_1$ is bounded ($|K_{\tilde{B}_1}|<e$), theorem \ref{CurvatureBound} imply that for $r_0e^{1/2}< \rho_1(\pi e^{-1},e)$ and $B_{r_0e^{1/2}}^{\tilde{\Sigma}}(p)\subset\hat\Sigma\setminus \partial \hat\Sigma$, for some $C$ we obtain for all $x\in B^\Sigma _{r_0 - \sigma}(p)$,

\begin{eqnarray}
 |A|^2(x)\leq e^{x_3-p_3}|\tilde{A}|^2(x)+1/2 \leq C e (e^{1/2}\sigma)^{-2}=C\sigma^{-2}.
\end{eqnarray}
\end{proof}

Using the curvature estimate and following the argument in Choi-Schoen \cite{CiSc}, we obtain following compactness theorem for translating graphs.

\begin{prop} \label{prop:comapctnessTS}
for fixed point $p$ in $\Real^3$, and $r >0$, let $\Sigma_j$ be embedded translating surfaces in
$B_{er}=B_{er}(p)\subset \Real^3$ with $\partial \Sigma_j \subset \partial B_{er}$. If each $\Sigma_j$ has area at most $V$
and genus at most $g$ for some fixed $V$, $g$ , then there is a finite collection of points $x_k$, a smooth embedded translating surface
$\Sigma \subset B_r$ with $\partial \Sigma\subset \partial B_r$ and a subsequence of the $\Sigma_j$'s that converges in $B_r$ (with finite multiplicity) to $\Sigma$ away from the $x_k$'s.
\end{prop}

\section{Classification of complete translating graphs}

In this section, we prove translating graphs over a domain $\Omega\subset\Real^2$ are asymptotic to a minimal surface next to the $\partial\Omega$. Which implies non existence of translating graphs over a bounded domain. Also, it concludes that complete translating graphs in $\Real^3$ can only be an entire graph over $\Real^2$ or be in one side of a vertical plane or between two vertical parallel planes.

\begin{lem}\label{delta}
If $\Sigma\subset\Real^3$ is a translating graph over a domain $\Omega\subset\Real^{2}$, then there is a $\delta>0$ such that for every $p\in\Sigma$, $\Sigma$ is a graph over the disk $D_\delta(p)\subset T_p\Sigma$ of radius $\delta$ centered at $p$.
\end{lem}
\begin{proof}



Let $p$ and $q$ be two different point in $\Sigma$. There is a geodesic $\gamma:[0,1]\to\Sigma$, so that $\gamma(0)=p$ and $\gamma(1)=q$. Now for $\upsilon$ normal vector to the $\Sigma$ we have
\begin{eqnarray}
|\upsilon(p)-\upsilon(q)|&\leq& \displaystyle\int_0^1 |\nabla_{\gamma'} \upsilon(\gamma(t))| dt \notag \\
&\leq&\displaystyle\int_0^1 |A(\gamma(t))| |\gamma'(t)| dt.
\end{eqnarray}

Hence the theorem \ref{CurvatureBound2} implies the lemma.

\end{proof}

\begin{thm}\label{t1}
There is no complete translating graph $\Sigma\subset\Real^{3}$ with nonzero constant speed $C$ over a bounded connected domain $\Omega\subset\Real^{2}$ with smooth boundary.
\end{thm}

\begin{proof}
The proof inspired by the one used in \cite{HA}. Suppose $\Sigma$ is a complete immersed translating graph over a domain $\Omega \subset \Real^2$. Lemma \ref{delta} implies there exists $\delta>0$ such that for each $p\in\Sigma$, $\Sigma$ is a graph in exponential coordinates over the disk $D_\delta(p)\subset T_p\Sigma$ of radius $\delta$, centered at the $p$. We denote this graph by $G(p)\subset\Sigma$, has bounded geometry. Note that $\delta$ is independent of $p$ and the bound on the geometry of $G(p)$ is uniform as well.

We define $F(p)$; the surface $G(p)$ translated to height zero $\Real^2=\Real^2\times \{0\}$, i.e, let $\alpha_p$ be the isometry of $\Real^{3}$ which takes $p$ to $\pi(p)$, we denote $F(p)=\alpha_p(\Sigma)$.

Now, let $p\in\Sigma$, since $\Sigma$ is a graph over $\Omega$, there is a function $u:\Omega\to \Real^3$ so that $\Sigma$ is the graph of $u$. If $\Sigma$ is not an entire graph then $\partial\Omega\neq\emptyset$. Since $\Sigma$ is a translating graph by mean curvature flow, $u$ has bounded gradient on relatively compact subsets of $\Omega$. Let $q\in\partial \Omega$ be such that $u$ does not extend to any neighborhood of $q$.

Let $q_n$ be a sequence in $\Omega$ converging to $q$, and let $p_n=(q_n,u(q_n))\in\Sigma$ be images of $q_n$ in $\Sigma$. Let $F_n$ denote the image of $G(p_n)$ under the vertical translation taking $p_n$ to $q_n$. Observe that $T_{q_n}(F_n)$ converges to the vertical plane $P$, for any subsequence of the $q_n$. Otherwise the graph of bounded geometry $G(p_n)$, would extend to a vertical graph beyond $q$, for $q_n$ close enough to $q$. Hence $f$ would extend; a contradiction.

For $q\in\Omega$, we define $L_\delta(q)$ a line of length $2\delta$ centered at $q$. Let $L_\delta(q)$ be the line whose normal vector has the same direction as the normal vector of limit normal vectors of $F_n$. Since each $F_n$ is a graph over $D_\delta(p_n)\subset T_{p_n}(F_n)$, the surfaces $F_n$ are bounded horizontal graphs over $L_\delta(q)\times[-\delta, \delta]$ for $n$ large. The compactness theorem for translating graphs implies that there is a subsequence of $F_n$'s which are converging to a translating surface $F$. The surface $F$ is tangent to $L_\delta(q)\times[-\delta, \delta]$ at $q$. Note that $F=L_\delta(q)\times[-\delta, \delta]$. Because if it is not the case there is a small $\epsilon>0$ so that $F(q-\epsilon \vec{n}(q))$ has two positive and negative values, where $\vec{n}$ is the unit normal to the $L_\delta(q)$. Therefore for $n$ large $F_n$ is not a graph, which is contradiction.

The plane $P=L_\delta(q)\times[-\delta, \delta]$, because both planes $P$ and $L_\delta(q)\times[-\delta, \delta]$ are passing through the point $q\in\partial\Omega$ and their normal vectors are the same.

In this point, we prove $u(q_n)\to +\infty$ or $u(q_n) \to -\infty$. Let $l$ be a line of length $\epsilon$ inside $\Omega$, starting at $q$, orthogonal to $\partial\Omega$ at $q$. Let $f$ be the graph of $u$ over $l$. At points near $q$, $l$ has no horizontal tangents, because tangent planes of $u$ at these points are converging to $P$. So we assume $u$ is increasing along $l$ as one converges to $q$. If $u$ is bounded above, then $f$ would have a finite limit point $(q,l_q)$ and $f$ would have finite length up till $(q,l_q)$. Since $\Sigma$ is complete, $(q,l_q)\in\Sigma$, which contradicts by $\Sigma$ has a vertical tangent plane at $(q,l_q)$.

Note that from Lemma \ref{lem:L-kernel}, $1/w$ satisfies an elliptic partial differential equation. Thus by the Harnack inequality, for any sequence $q_n\in \Omega$ converging to $q$ we have $w(q_n)\rightarrow +\infty$. That means $H(q_n)\rightarrow 0$.

Which is contradiction, since the domain is bounded, the mean curvature of the graph next to the boundary should converge to the mean curvature of the cylinder $\partial\Omega\times\Real$, which is not zero.
\end{proof}

\begin{cor}
 If $\Sigma$ is a complete translating graph over a domain $\Omega\subset\Real^2$, then next to the boundary of $\Omega$, $\Sigma$ is asymptotic to a plane. So a translating graph over $\Real^2$ can only be between 2 parallel planes or in one side of a plane or an entire graph.
 \end{cor}

 \begin{proof}
 From the proof of Theorem \ref{t1}, next to the boundary of $\Omega$, the graph $\Sigma$ converges to a minimal surface. Since $\Sigma$ is complete, it can only converge to a vertical plane.
 \end{proof}

\bibliographystyle{spmpsci}

\end{document}